\documentclass[12pt]{elsarticle}
\usepackage{float}
\usepackage{upgreek}
\usepackage[table]{xcolor}

\usepackage{amssymb}
\usepackage{graphicx}
\usepackage{epstopdf}
\usepackage{latexsym}
\usepackage{amsmath}
\usepackage{amsfonts}
\usepackage{array}
\usepackage{microtype}
\usepackage{amsthm}
\usepackage{hyperref}
\newtheorem{definition}{\bf Definition}[section]
\newtheorem{lemma}[definition]{\bf Lemma}
\newtheorem{theorem}[definition]{\bf Theorem}
\newtheorem{proposition}[definition]{\bf Proposition}

\newtheorem{remark}[definition]{\bf Remark}

\usepackage{hyperref}
\usepackage{xcolor}

\def\N{\ensuremath{\mathbb{N}}}

\def\C{\ensuremath{\mathbb{C}}}

\newcommand{\HC}{\mathcal{H}({\mathbb{C}})}

\begin{document}

\begin{frontmatter}

\title{Supercyclic properties of extended eigenoperators of the differentiation operator on the space of entire functions}

\author{Manuel Gonz\'alez}
\address{Department of Mathematics, 
University of Cantabria, Avenida de los Castros s/n, 39005-Santander,  Spain.} \ead{manuel.gonzalez@unican.es}

\author{Fernando Le\'on-Saavedra}
\address{ Department of Mathematics, 
University of C\'adiz, Avda. de la Universidad s/n,  11402-Jerez de la Frontera, Spain.} \ead{fernando.leon@uca.es}

\author{Mar\'{\i}a Pilar Romero de la Rosa}

\address{Department of Mathematics, University of C\'adiz, CASEM, Pol. R\'{\i}o San Pedro s/n, 11510-Puerto Real,  Spain.}
\ead{pilar.romero@uca.es}

\begin{abstract}
A continuous linear operator $L$ defined on the space of entire functions $\HC$ is said to be an extended $\lambda$-eigenoperator of the differentiation operator $D$ provided $DL=\lambda LD$. Here we fully characterize when an extended $\lambda$-eigenoperator of $D$ is supercyclic, it  has a hypercyclic subspace or  it has a supercyclic subspace.

\end{abstract}


\end{frontmatter}

\pagenumbering{arabic}

\section{Introduction}

A (continuous, linear) operator $T$ acting on a separable Fr\'echet space $\mathcal{X}$ is said to be {\it hypercyclic} if there exists $f\in \mathcal{X}$ such that $\{T^nf\}_{n\geq 1}$ is dense in $\mathcal{X}$. 
 In this case we say that the vector $f$ is hypercyclic for $T$. To say that $T$ is {\it supercyclic} means that there exists a vector $x\in\mathcal{X}$ such that the projective orbit $\{\lambda T^nx\,\,:\,\,n\in\mathbb{N},\lambda\in\mathbb{C}\}$ is dense in $\mathcal{X}$.
 The  operator $T$ is said to have a  {\it hypercyclic (supercyclic) subspace} if there exists an infinite dimensional closed subspace  whose non-zero vectors are hypercyclic (supercyclic) for $T$. The study of hypercyclic (supercyclic) subspaces is a mainstream research line in the theory of hypercyclic operators (see \cite{erdmannperis} Chapter 10, \cite{bayartmatheron} Chapter 8,  \cite{montessalas}), and its interest come from the invariant subspace problem. 
 
 Over a century ago,  Birkhoff \cite{birkhof} proved that the translation operator is hypercyclic on $\HC$ endowed with the compact open topology. Later, MacLane \cite{maclane} proved that the differentiation operator is also hypercyclic. These seem to be the first examples of hypercyclic operators. Godefroy and Shapiro
 \cite{godefroyshapiro} unified the results of Birkhoff and MacLane by showing that each non-scalar operator $A$ (i.e., $A\neq \lambda I$) commuting with $D$ is hypercyclic. 
 
An operator $T$ is said to {\it  $\lambda$-commute with $A$} provided $AT=\lambda TA$, where  $\lambda \in \mathbb{C}$. If $T\neq 0$, then $T$ is called an {\it extended $\lambda$-eigenoperator} of $A$ with {\it extended eigenvalue} $\lambda$. 
Extended $\lambda$-eigenoperators   appear in the proof of an extension of Lomonosov's famous result, which was obtained independently by Scott Brown and Kim, Moore and Pearcy \cite{inv1,inv2}. They proved that if an operator has a compact extended $\lambda$-eigenoperator then it has a non-trivial hyperinvariant subspace.
The determination of the extended $\lambda$-eigenoperators plays an important role in the study of an operator, and some properties, like hypercyclicity, are transferred in some way to the commutant (see  \cite{godefroyshapiro}). 

In this paper we study how the hypercyclic properties are transferred to the $\lambda$-commutant. Some results in \cite{jfa} reveal  that the hypercyclic properties for the extended $\lambda$-eigenoperators of $D$ enjoy a rich structure.

For the differentiation operator $D$ on $\HC$, Bonilla and Grosse-Erdmann studied how frequent hypercyclicity is transferred to the commutant of $D$  (see \cite{bonilla}).  Recall that an operator $T$ on a Fréchet space $\mathcal{X}$ is said to be \emph{frequently hypercyclic} if there exists a vector $x\in \mathcal{X}$ such that for each nonempty open subset $U$ of $\mathcal{X}$ the set
$N_U=\{n\in\mathbb{N}\,: \,T^nx\in U\}$
has lower density in $\mathbb{N}$.
The study in \cite{freq} reveals that  frequent hypercyclicity is not always  transferred to the $\lambda$-commutant. 

It has also been studied when an operator that commutes with $D$ has an   hypercyclic subspace. S. Shkarin \cite{shkarin} began by proving that $D$ has a hypercyclic subspace. 
Partial results were obtained by H. Petterson \cite{petersson} and completed by Q. Mennet \cite{menet}, obtaining the so called  Petersson-Menet-Shkarin's result which asserts that each non-scalar operator commuting with $D$ has a hypercyclic subspace. Thus the following question arises:

\begin{quote}
{\bf Question 1. } Assume that $L$ is an extended $\lambda$-eigenoperator of $D$. When does $L$ have a  hypercyclic subspace?
\end{quote}

In strong contrast with Petersson-Menet-Shkarin's result, an extended $\lambda$-eigenoperator $L$ of $D$ does not always have a hypercyclic subspace. Indeed, if we consider the following 
extended $\lambda$-eigenoperators of the differentiation operator: 
$T_{\lambda,b}f=f'(\lambda z+b)$, introduced by Aron-Markose \cite{aron}, it was shown in  \cite{pilar}  that $T_{\lambda,b}$ has a hypercyclic subspace if and only if $|\lambda|=1$.

It is known from \cite{jfa} that if $L$ is an extended $\lambda$ eigenoperator of $D$ then $L$ factorizes as $L=R_\lambda\phi(D)$ where $R_\lambda f(z)=f(\lambda z)$ is the dilation operator and $\phi$ is an entire function of exponential type, that is, there exist constants $A,B$ such that $|\phi(z)|\leq A e^{B|z|}$ for all $z\in \mathbb{C}$. 

The results in \cite{pilar} may suggest that  extended $\lambda$-eigenoperators  with $|\lambda|>1$ have no hypercyclic subspace. For  $T_{\lambda,b}$ this is true because in the factorization $T_{\lambda,b}=R_\lambda \phi(D)$ the map $\phi(z)=ze^{bz}$ has only one zero. However the situation is quite different when $\phi$ has infinitely many zeros. Our first main result fully answer Question 1.
\medskip

{\bf Main result 1.} {\it Assume that $L=R_\lambda \phi(D)$ is an extended $\lambda$-eigenoperator of $D$. The following conditions are equivalent:
\begin{enumerate}
\item $L=R_\lambda\phi(D)$ has a hypercyclic subspace.
\item $\phi$ has infinitely many zeros and $|\lambda|\geq 1$ or $\phi$ has a finite (non empty) number of zeros and $|\lambda|=1$.
\end{enumerate}
}

The proof of this result splits into several cases. 
If we denote $\mathcal{Z}(\phi)$ the set of zeros of $\phi$ and $HC_\infty$ the set of all operators having a hypercyclic subspace, Table \ref{tabla1} below summarizes the results.

\begin{table}[h]
\centering
    \begin{tabular}{|c|c|c|c|}
    \hline
    $L=R_\lambda\phi(D)$ & $\mathcal{Z}(\phi)=\emptyset$  & $\mathcal{Z}(\phi)\neq \emptyset$ finite  & $\mathcal{Z}(\phi)$ infinite\\
    \hline

     $|\lambda|<1$ &\cellcolor{gray!25}  $L\notin HC_\infty$  \cite{jfa}&   \cellcolor{gray!25} $L\notin HC_\infty$ \cite{jfa} & \cellcolor{gray!25} $L\notin HC_\infty$   \cite{jfa}  \\
     \hline
     $|\lambda|=1$ & \cellcolor{gray!25}$L\notin HC_\infty$ \cite{jfa}   & $L\in HC_\infty$; Th. \ref{modulo1}& $L\in HC_\infty$;  Prop. \ref{infinitosceros} \\
     \hline
     $|\lambda|>1$ &  \cellcolor{gray!25} $L\notin HC_\infty$ \cite{jfa}    &  $L\notin HC_\infty$; Th. \ref{modulomayoruno} & $L\in HC_\infty$; Prop. \ref{infinitosceros}    \\
     \hline
    \end{tabular}
\caption{Hypercyclic subspaces}
    \label{tabla1}
\end{table}

The shaded entries in Table \ref{tabla1} correspond with results discovered in \cite{jfa}. In those cases the operator $L=R_\lambda\phi(D)$ is not hypercyclic, therefore it has no hypercyclic subspaces.

Moreover, we devote some effort to sharpen the results of Table \ref{tabla1}. Let us denote by $SC_\infty$ the set all supercyclic operators having a supercyclic subspace. For $L\notin HC_\infty$, we want to know whether $L\in SC_\infty$, and we begin studying the supercyclicity of the extended $\lambda$-eigenoperators of $D$. 

\begin{quote}
{\bf Question 2.} Let $L$ be a non-hypercyclic extended $\lambda$-eigenoperator of $D$. Is $L$  supercyclic?
\end{quote}


\begin{quote}
{\bf Question 3.} Let $L$ be a supercyclic extended $\lambda$-eigenoperator of $D$. Does $L$ have a  supercyclic subspace?
\end{quote}

When $\phi$ has no zeros, then the operator $L=R_\lambda\phi(D)$ is a composition operator $C_{\lambda,b}f(z)=f(\lambda z+b)$
induced by an affine endomorphisms. 
In this case, it was proved by Bernal-González, Bonilla and Calderón-Moreno \cite{bernalbonillacalderon} that $L$ is not supercyclic (see the shaded entries in Table \ref{tabla2}). 

For Question 2 it is enough to study the case $|\lambda|<1$ because in other cases the operator is hypercyclic. 
We fully answer this question and our results sheds light on an additional, unexpected condition: The value of $\phi$ at the origin. When $\phi(0)=0$ we can obtain a positive result using standard arguments. However,
the case $\phi(0)\neq 0$, requires more efforts and  new ideas. In general, to prove that an operator is not supercyclic is much more complicated than proving that it is. Denoting by $SC$ and $HC$ the set of all supercyclic and hypercyclic operators respectively, we obtain:
\medskip

{\bf Main result 2. }{\it Assume that $L=R_\lambda \phi(D)$ is an extended $\lambda$-eigenoperator of $D$. The following conditions are equivalent:
\begin{enumerate}
    \item $L\in SC\setminus HC$.
    \item $|\lambda|<1$ and $\phi(0)=0$.
\end{enumerate}
}

Again the proof splits into two  cases, and Table \ref{tabla2} summarizes the results.

\begin{table}[h]
    \centering
    \begin{tabular}{|c|c|c|c|}
\hline
    $L=R_\lambda\phi(D)$ & $\mathcal{Z}(\phi)=\emptyset$  & $\phi(0)=0$ & $\mathcal{Z}(\phi)\neq \emptyset$, $\phi(0)\neq 0$ \\
    \hline
     $|\lambda|<1$ & \cellcolor{gray!25} $L\notin SC$ \cite{bernalbonillacalderon}  &    $L\in SC\setminus HC$; Th. \ref{super1} & $L\notin SC$;  Th. \ref{super2}   \\
     \hline
    \end{tabular}
    \caption{Supercyclicity}
    \label{tabla2}
\end{table}

 We also study when an extended $\lambda$-eigenoperator has a supercyclic subspace. The remaining cases are: 1) $|\lambda|<1$ and $\phi(0) =0$ (Theorem \ref{modulomenor1}), in which there exists a supercyclic subspace, and 2) $|\lambda |> 1$ and $\phi(D)$ has a non-emtpy finite number of zeros, in which  the operator does not have a supercyclic subspace. The proof uses a mix between the ideas of \cite{menet} and \cite{montessalas}, however these ideas separately seem not enough to obtain the desired result. In summary: 
\medskip

{\bf Main result 3.} {\it Let $L$ be an extended $\lambda$-eigenoperator of $D$, $\lambda\neq 1$. The following conditions are equivalent:
\begin{enumerate}
\item $L\in SC_\infty\setminus HC_\infty$.
\item $0<|\lambda|<1$ and $\phi(0)=0$.
\end{enumerate}
}

Table \ref{tabla3} summarizes the results and points at the places in the paper where those results can be found.. Here the shaded entries in the first column follow by the non-supercyclicity of the operator (\cite{bernalbonillacalderon}), and the other shaded  entries follow from the study on hypercyclic subspaces (see Table \ref{tabla1}).

    \begin{table}[h]
    \centering
    \begin{tabular}{|c|c|c|c|}
    \hline
    $|\lambda|<1$ & $\mathcal{Z}(\phi)=\emptyset$  & $0\in \mathcal{Z}(\phi)$   & $0\notin \mathcal{Z}(\phi)$ \\
    \hline
      & \cellcolor{gray!25}$L\notin SC_\infty$ \cite{bernalbonillacalderon} &    $L\in SC_\infty\setminus HC_\infty$; Th. \ref{modulomenor1}  & \cellcolor{gray!25} $L\notin SC_\infty$  (Th. \ref{super2})   \\
     \hline
     $|\lambda|=1$ & $\mathcal{Z}(\phi)=\emptyset$  & $\mathcal{Z}(\phi)\neq \emptyset$ finite  & $\mathcal{Z}(\phi)$ infinite\\
     \hline
     & \cellcolor{gray!25}$L\notin SC_\infty$  \cite{bernalbonillacalderon} & \cellcolor{gray!25}$L\in HC_\infty$; Th. \ref{modulo1}  & \cellcolor{gray!25}$L\in HC_\infty$ ; Prop.\ref{infinitosceros}  \\
     \hline
     $|\lambda|>1$ & $\mathcal{Z}(\phi)=\emptyset$  & $\mathcal{Z}(\phi)\neq \emptyset$ finite  & $\mathcal{Z}(\phi)$ infinite\\
     \hline
      & \cellcolor{gray!25} $L\notin SC_\infty$   \cite{bernalbonillacalderon}  &  $L\notin SC_\infty$;  Th.\ref{general} & \cellcolor{gray!25}$L\in HC_\infty$; Prop.\ref{infinitosceros}    \\
     \hline
    \end{tabular}
    \caption{Supercyclic subspaces}
    \label{tabla3}
\end{table}

The main ancestors of this paper are  \cite{gonzalezleonmontes} and \cite{leonmuller}. Many of the results in these two papers were extended to operators on Fréchet spaces in \cite{menet}. 

The paper is structured as follows. In Section \ref{secciondos} we describe some tools from \cite{petersson} and \cite{menet}  that we will use throughout the paper.
In Section \ref{secciontres} we prove Main result 1. The case when $\phi$ has finitely many zeros is reduced to the case of a polynomial $P$ by showing that the operators $R_\lambda\phi(D)$ and $R_\lambda P(D)$ are  similar. In Section \ref{seccioncuatro} we characterize when $R_\lambda \phi(D)$ is supercyclic in terms of an intriguing and unexpected condition on the values of $\phi(0)$.  Finally, in Section \ref{seccioncinco} we prove the Main result 3. 
\medskip

{\bf Acknowledgement. } The authors were supported by Ministerio de Ciencia, Innovaci\'on y Universidades (Spain), grants PGC2018-101514-B-I00,  PID2019-103961GB-C22, and by  Vicerrectorado de Investigaci\'on de la Universidad de C\'adiz.

\section{Some preliminary results}

\label{secciondos}

Here we introduce some tools  we will use along the paper. Let $\mathcal{X}$ be a separable Fr\'echet space. Next result is a version of the Hypercyclicity Criterion discovered by Bès-Peris (see Theorem 3.24 in \cite{erdmannperis}).

\begin{theorem}[Hypercyclicity criterion for sequences]
\label{HPforseq}
Let $(T_n)$ be a sequence of operators acting on $\mathcal{X}$. Suppose that there exist two dense subsets $X_0, Y_0$ of $\mathcal{X}$, a subsequence $(T_{n_k})$ and a sequence of maps $S_{k}:Y_0\to\mathcal{X}$ satisfying:
\begin{description}
\item[i)] $T_{n_k} x_0\to 0$ for all $x_0\in X_0$.
\item[ii)] $S_ky_0\to 0$ and 
\item[iii)] $T_{n_k}S_ky_0\to y_0$ for all $y_0\in Y_0$.
\end{description}
Then there exists $x\in \mathcal{X}$ such that the set $\{T_{n_k}x : k\in\N \}$ is dense in $\mathcal{X}$.
\end{theorem}

We say that an operator $T$ on $\mathcal{X}$ \emph{satisfies the Hypercyclicity Criterion for the sequence $(n_k)$} if the hypothesis of Theorem \ref{HPforseq} holds for the sequence of powers $(T^n)$, and the sequence $(n_k)$. 

To obtain an infinite dimensional closed subspace of hypercyclic vectors, we need a stronger  Hypercyclicity Criterion involving an infinite dimensional closed subspace $M_0$ where we can control the orbits. Such a result appeared firstly in \cite{gonzalezleonmontes} in the Banach space setting. Here we will use the following Fréchet space extension by Petersson  \cite{petersson}.

\begin{theorem}
\label{yes}
Let $(T_n)$ be a sequence of operators on a Fr\'echet space $\mathcal{X}$ with a continuous norm. If $(T_n)$ satisfies the Hypercyclicity Criterion for a  subsequence $(n_k)$ of $\N$ and there exists an infinite dimensional closed subspace $M_0\subset \mathcal{X}$ such that $T_{n_k}f\to 0$ for all $f\in M_0$, then there is an infinite dimensional closed subspace $M_1$ such that for every $x\in M_1\setminus \{0\}$ the orbit $\{T_{n_k}x\}$ is dense in $\mathcal{X}$.
\end{theorem}

The existence of the subspace $M_0$ in the previous Theorem can be weakened, requiring only a control of the sequence of operators in a sequence of closed subspaces of decreasing infinite dimension. This weakening was proved firstly in the Banach space setting in  \cite[Theorem 20]{leonmuller}. We will use the following extension for Fréchet spaces discovered by Q. Menet (see \cite{menet}, Theorem 1.11 and Remark 1.12).

\begin{theorem}
\label{yes2}
Let $T$ be an operator on $\mathcal{X}$  satisfying the Hypercyclicity Criterion for $(n_k)$, and let $(\rho_n)$ be an increasing sequence of seminorms defining the topology of $\mathcal{X}$. If there exists a decreasing sequence of infinite dimensional closed subspaces $(M_j)$ such that, for each $n\in\N$, we can find   $C_n>0$ and $m(n), k(n)\in\N$ so that for each $j\geq k(n)$ and $x\in M_j$,
$$
\rho_n(T^{n_j}x)\leq C_n \rho_{m(n)}(x),
$$
then $T$ has a hypercyclic subspace.
\end{theorem}

\begin{remark}
\label{seq}
Theorem \ref{yes2} remains true replacing the sequence $(T^{n_k})$ by a sequence of operators $(T_{n_k})$ (see \cite[Theorem 1.11 and Remark 1.12]{menet}).
\end{remark}

To prove that a hypercyclic operator on Banach spaces does not have a hypercyclic subspace, it suffices to inspect its essential spectrum (see \cite{gonzalezleonmontes}). For sequences of operators on Banach spaces, the result by González-León-Montes (\cite{gonzalezleonmontes}) was refined by León-Muller in  \cite{leonmuller} (Theorem 22 and Corolary 23). For sequences of operators defined on Fréchet spaces this problem is more delicate and much more complex. The following sufficient condition to prove the nonexistence of hypercyclic subspaces was discovered by Q. Menet in \cite{menet} (Theorem 2.2 and Corollary 2.3).

\begin{theorem}
\label{not}
Let $T$ be an operator on a Fr\'echet space $\mathcal{X}$ with a continuous norm. Assume that there exist a sequence of seminorms $(\rho_n)$ defining the topology of $\mathcal{X}$ and $N\geq 1$ such that for every $n\in\N$, there exist $C_n>1$, a closed subspace $N_n$ of finite codimension  such that
$$
\rho_N(T^{n}x)\geq C_n \rho_n(x) \textrm{ for $x\in N_n$.}
$$
Then $T$ has no hypercyclic subspace.
\end{theorem}


\section{Hypercyclic subspaces for extended eigenoperators of $D$}
\label{secciontres}

We begin  by describing some results obtained in \cite{jfa} where it was characterized when  an extended $\lambda$-eigenoperator  of $D$ is hypercyclic. First of all, the extended $\lambda$-eigenoperators of $D$ were described as follows:

\begin{proposition}
\label{factorizacion}
\emph{(\cite[Proposition 2.3]{jfa})}
Let $L$ be an operator on $\HC$. 
Then $DL=\lambda LD$ for some $0\neq\lambda\in\C$ if an only if
$L=R_\lambda\phi(D)$, 
with $R_\lambda f(z)=f(\lambda z)$ for $z\in \mathbb{C}$ and $\phi$ an entire function of exponential type.
\end{proposition}

From Proposition \ref{factorizacion}, it was derived  that if $\phi$ has no zeros, then the operator $R_\lambda \phi(D)$ is a multiple of a composition operator induced by an affine endomorphisms: $C_{\lambda,b}f(z)=f(\lambda z+b)$. The  main result in \cite{jfa} characterizes when $R_\lambda \phi(D)$ is hypercyclic:

\begin{theorem}
\label{extended}
\emph{(\cite{jfa})}
For an operator $L$ on $\HC$ satisfying $DL=\lambda LD$, $\lambda\neq 1$, the following conditions are equivalent:
\begin{enumerate}
\item $L$ is hypercyclic.
\item $L$ satisfies the Hypercyclicity Criterion.
\item $|\lambda|\geq 1$ and $L$ is not a multiple of the  operator $C_{\lambda,b}f(z)=f(\lambda z+b)$.
\end{enumerate}
\end{theorem}

Now we can restate our first main result as follows:
\medskip{}{}

{\bf Main result 1. }{\it Let  $L=R_\lambda \phi(D)$ be an extended $\lambda$-eigenoperator of $D$. Then the following conditions are equivalent:
\begin{enumerate}
    \item $L$ has an hypercyclic subspace.
    \item $\phi$ has infinitely many zeros and $|\lambda|\geq 1$, or $\phi$ has a non-empty finite number of zeros and $|\lambda|=1$.
\end{enumerate}
}

The proof splits into several cases which we will study separately. Next we  consider the case when $\lambda$ is a root of the unity.

\begin{theorem}
\label{raizunidad}
Assume that $\lambda$ is a root of the unity. If $DL=\lambda LD$ and $L$ is not a multiple of $C_{\lambda,b}$ then $L$ has a hypercyclic subspace.
\end{theorem}
\begin{proof}
If $\lambda^{n_0}=1$ for some $n_0\in\mathbb{N}$ then  $R_{\lambda}^{n_0}=I$. Since
$L=R_\lambda\phi(D)$ is not a multiple of  $C_{\lambda,b}$, $\phi$ has some zero, $L^{n_0}$  is not a multiple of the identity and $L^{n_0}$ commutes with $D$.
Thus, since by Menet-Petterson-Shkarin's result (see \cite{menet,petersson,shkarin}) each non-scalar operator commuting with the differentiation operator has a hypercyclic subspace, we get that $L^{n_0}$ has a hypercyclic subspace. Hence $L$ has a  hypercyclic subspace as we desired.
\end{proof}

\begin{proposition}
\label{infinitosceros}
Let  $L=R_\lambda \phi(D)$ be an extended $\lambda$-eigenoperator of $D$. If $L$ is hypercyclic and $\phi$ has infinitely many zeros then $L$ has a hypercyclic subspace.
\end{proposition}
\begin{proof}
By Theorem \ref{extended}, $L$ satisfies the Hypercyclicity Criterion for some subsequence $(n_k)$ of $\mathbb{N}$.
 If $a$ is a zero of $\phi$ then $Le^{az}=\varphi(a) e^{a\lambda z}=0$.
Since $\phi$ has infinitely many zeros, $\textrm{Ker}(L)$ is infinite dimensional. Since the sequence  $(L^n)$  converges trivially  to zero on $\textrm{Ker}(L)$, by Theorem \ref{yes} $L$ has a hypercyclic subspace as we desired.
\end{proof} 



Proposition \ref{infinitosceros} reduces our problem to the case that  $\phi$ is an entire function of exponential type which only has  finitely many zeros. 
By Hadamard factorization Theorem, $\phi(z)=P(z)e^{bz}$, where $P(z)$ is a polynomial. This case includes the operators $T_{\lambda,b}$ studied in \cite{aron}. 
By Theorem \ref{extended}, it is sufficient to study the case  $|\lambda|\geq 1$.
Since the existence of a hypercyclic subspace is invariant under similarity, our next result simplifies the problem.

\begin{proposition}
\label{semejanza}
Set $\lambda\neq 1$. The operator $R_\lambda P(D)e^{bD}$ is  similar to $R_\lambda P(D)$. 
\end{proposition}
\begin{proof}
Set $\alpha=\frac{b}{1-\lambda}$. 
%
%
It is easy to check that $e^{-\alpha D}R_\lambda = R_\lambda e^{-\lambda \alpha D}$. Hence 
$$
e^{\alpha D}R_\lambda P(D)e^{\alpha D}= R_\lambda e^{-\lambda \alpha D} P(D) e^{\alpha D}= R_\lambda P(D)
e^{\alpha (1-\lambda)D}= R_\lambda P(D)e^{bD}, 
$$
and the result is proved.
%
%
\end{proof}

For $f(z)=\sum_{k=0}^\infty a_k z^k \in\HC$, we define 
$\rho_M(f)=\sum_{k=0}^\infty |a_k| M^k$.
The family of seminorms $\{\rho_M : M>0\}$ induces the natural topology of $\HC$. 

The case $\lambda=1$ of the following result was proved by Menet \cite{menet}. 

\begin{theorem}
\label{modulo1}
If $L=R_\lambda P(D)$ with $|\lambda|=1$ and $P$ is a non-constant polynomial, then $L$ has an hypercyclic subspace.
\end{theorem} 
\begin{proof}
By Theorem \ref{extended}, $L=R_\lambda P(D)$ satisfies the Hypercyclicity Criterion for some subsequence  $(n_k)$, (moreover by \cite{freq}, the operator $L$  satisfies the Hypercyclicity Criterion for the full sequence of natural numbers).
 
Let us denote $\omega=\lambda^{-1}$ and  $P(z)= \sum_{k=0}^d b_kz^k$, for each $n\in\N$ we write 
$$
P(\omega z)P(\omega^2 z)\cdots P(\omega^{n}z) =\sum_{k=0}^{nd}b_k^{(n)} z^k
$$ 
and set $\widetilde{C_n}=\max\{|b_k^{(r)}|\,:\, 1\leq r\leq n\,,\, 1\leq k\leq dr\}$. 
Then
\begin{eqnarray*}
(R_\lambda P(D))^n (z^s)&=& R_\lambda P(D)\cdots R_{\lambda} P(D) (z^s) \\
&=& P(\omega D)\cdots P(\omega^n D) R_{\lambda}^n (z^s) \\
&=& \sum_{k=0}^{nd} b_k^{(n)} [s(s-1)\cdots (s-k+1)] \lambda^{ns} z^{s-k}.
\end{eqnarray*}

Thus, if $|z|=M\geq 1$,
$$
|(R_\lambda P(D))^n (z^s)|\leq  \sum_{k=0}^{nd} |b_k^{(n)}| [s(s-1)\cdots (s-k+1)]  
M^{s-k} 
\leq  \widetilde{C_n} nd  s^{nd} M^s.
$$

Clearly, for each $n\in\N$ we can select $M_n>0$ such that 
$\widetilde{C_n} nd x^{dn}\leq 2^x$
for $x\geq M_n$. We consider a strictly increasing sequence $(n_j)$ in $\N$ with $n_{j+1}\geq M_{n_j}$. Thus if $s\geq j+1$ then 
\begin{equation}
\label{desigualdad}
    \widetilde{C_{n_j}} n_j d n_s^{n_jd} \leq 2^{n_s}.
\end{equation}

If $f\in N_j=\overline{\textrm{linearspan}} \{z^{n_s}\,: \,s\geq j+1\}$, then $f(z)= \sum_{s=j+1}a_s z^{n_s}$ and
\begin{eqnarray*}
\rho_M(L^{n_j} f)&=& \rho_M\left(L^{n_j}\left(\sum_{s=j+1}^{\infty} a_s z^{n_s}\right) \right)\\
&=& \rho_M\left(\sum_{s=j+1}^{\infty} a_s L^{n_j}z^{n_s}\right) \\
&\leq& \sum_{s=j+1}^\infty |a_s| n_jd \widetilde{C_{n_j}} n_s^{n_j} M^{n_s}\\
&\leq &  \sum_{s=j+1} |a_s| (2M)^{n_s}=  \rho_{2M}(f), 
\end{eqnarray*}
where the last inequality follows from equation (\ref{desigualdad}).  Thus Theorem \ref{yes2} is fulfilled if we consider the sequence of seminorms $(\rho_n)$ and we take $m(n)=2n$, $C_n=k(n)=1$. Thus if $p(D)$ is not a multiple of the identity and $|\lambda|=1$ then $L=R_\lambda P(D)$ has a hypercyclic subspace, as we desired to prove.
\end{proof}

The most intriguing case is when $|\lambda|>1$ and $\phi$ has a finite number of zeros; or just  $\phi (D)=a_dD^d+\cdots a_0 I$ is a polynomial. The main idea of the proof is basically to show that the significant term of the action  of $L^n$ (on finite codimensional subspaces) is concentrate on $a_dD^d$. If we set
$$
N_n=\{f\in \HC\,:\,f(0)=f'(0)=\cdots=f^{(n)}(0)=0\},
$$
by bounding the derivative on a line, it is simple to show that there exist $c>0$ and $r>0$ such that for any $f\in N_d$
$$
\rho_r(Lf)\geq c |\lambda|^d \rho_r(f^{(d)}(\lambda z)).
$$
However this idea of boundedness is not enough for the asymptotic inequality that is needed. We stress here that the following asymptotic formula relies on taking finite codimensional subspaces $N_n$ for $n$ large enough.

\begin{lemma}
\label{iteracionpolinomio}
Suppose that $L=R_\lambda P(D)$ with $P$ a polynomial of degree $d\geq 1$. Set $M>0$. There exist  a constant $c>0$  and an increasing sequence $(m_n)\subset \mathbb{N}$ such that for each $n\in\N$ and $h\in N_{m_n}$ 
\begin{equation}
\label{formulaasintotica}
\rho_M(L^nh)\geq c^n 1\cdot |\lambda|^d\cdot |\lambda|^{2d}\cdots |\lambda|^{(n-1)d} \rho_M(h^{(nd)}(\lambda^n z)).
\end{equation}
\end{lemma}
\begin{proof}
We can assume that $L=P(D) R_{\lambda}$ and $P(D)= \sum_{k=0}^d a_kD^k$. 

For $n=1$, set $c>0$ satisfying $|a_d|>c$, and let  $h\in N_m$ with $m$ (large enough) to be determined later. We write 
$$
\rho_M(Lh)\geq \rho_M(a_dD^dR_\lambda h)-\sum_{k=0}^{d-1}|a_k|\rho_M(D^kR_\lambda h).
$$

If $h(\lambda z)=\sum_{n=m}^\infty h_nz^n$, let us check the above formula in each  $z^p$ for $p\geq m$. Since 
$$
\xi_{M,d}(p)=\frac{|a_0|M^d+|a_1|pM^{d-1}+\dots+|a_{p-1}|p\cdots (p-d+2) M }{p(p-1)\cdots (p-d+1)}\to 0
$$
as $p\to \infty$, there exists $m_1>0$ such that for $p\geq m_1$,  $|a_d|-\xi_{M,d}(p)>c$. Thus
$$
\frac{p\cdots (p-d+1)}{M^d} \left[|a_d| |h_p|M^{p}- \xi_{M,d}(p) |h_p|M^p\right] \geq c p\cdots (p-d+1)|h_p|M^{p-d},
$$
and for each $f\in N_{m_1}$ we get 
\begin{eqnarray*}
\rho_M(Lh)&\geq &\rho_M(a_dD^dR_\lambda h)-\sum_{k=0}^{d-1}|a_k|\rho_M(D^kR_\lambda h)\\
&\geq& c \rho_M(D^dh(\lambda z))=c |\lambda|^d \rho_M(h^{(d)}(\lambda z)).
\end{eqnarray*}
The proof for arbitrary $n$ is similar: if $L^n=p(D)p(\omega D) \cdots p(\omega^{n-1}D) R_{\lambda}^n$, then
\begin{eqnarray*}
\rho_M(L^nh)&=&\rho_M(p(D)p(\omega D)\cdots p(\omega^{n-1}D) R_{\lambda}^nh)\\ 
&=& 1\cdot |\omega|^d\cdots |\omega|^{d(n-1)}\rho_M(\widehat{P}(D)R_{\lambda}^nh),
\end{eqnarray*}
where $\widehat{P}(D)$ is a  polynomial whose leader term is $|a_d|^nD^{nd}$. Thus, by arguing as in the first step, there exists $m_n$ such that for any $h\in N_{m_n}$ we get
\begin{eqnarray*}
\rho_M(L^nh)&=& 1\cdot |\omega|^d\cdots |\omega|^{d(n-1)}\rho_M(\widehat{P}(D)R_{\lambda}^nh)\\
&\geq & c^n1\cdot |\omega|^d\cdots |\omega|^{d(n-1)}\rho_M(D^{nd}R_{\lambda}^nh)\\
&=& c^n1\cdot |\omega|^d \cdots  |\omega|^{d(n-1)}    |\lambda|^{nd}\cdots |\lambda|^{nd} \rho_M(h^{(nd)}(\lambda^nz))\\
&=& c^n|\lambda|^d|\lambda|^{2d} \cdots |\lambda^{nd}|\rho_M(h^{(nd)}(\lambda^nz)),
\end{eqnarray*}
and the result is proved.
\end{proof}

\begin{remark}
\label{observacionpolinomio}
For future reference, we remark that the asymptotic inequality
(\ref{formulaasintotica}) remains true for any subsequence $(r_n)$ of $(m_n)$. This fact will be crucial in the proof of Theorem \ref{general}.
\end{remark}

The following result completes the proof of   Main result 1. 
Surprisingly there are hypercyclic extended $\lambda$-eigenoperators of $D$ such that all closed subspaces of hypercyclic vectors have finite dimension.

\begin{theorem}
\label{modulomayoruno}
If $L=R_\lambda \phi(D)$ with $|\lambda|>1$ and $\phi$ has finitely many zeros, then each  closed subspace of hypercyclic vectors for $L$ is finite dimensional.
\end{theorem}
\begin{proof}

We can assume that $L=P(D)R_\lambda$ with $P(z)=\sum_{k=0}^db_kz^k$  non-constant. By applying Lemma \ref{iteracionpolinomio} there exist  a constant $c>0$  and an increasing sequence $(m_n)\subset \mathbb{N}$ with $m_n\geq nd$ such that for any $h\in N_{m_n}$
\begin{equation}
\label{nuevotres}
\rho_1(L^nh)\geq c^n 1\cdot |\lambda|^d\cdot |\lambda|^{2d}\cdots |\lambda|^{(n-1)d} \rho_1(h^{nd)}(\lambda^n z)).
\end{equation}

Let us consider the seminorms $p_0(h)=\max_{|z|\leq 1} |h(z)|$, and 
$$p_n(h)=\max_{|z|\leq |\lambda|^{n/4}}|h(z)|.$$
Then $(p_n)$ is an increasing sequence of seminorms that defines the original topology of $\HC$. 
Moreover, since $m_n\geq nd$, for  $h\in N_{m_n}$ we get
\begin{eqnarray*}
\rho_{1}(h^{(nd)}(\lambda^nz)&=& \max_{|z|\leq 1} |h^{(nd)}(\lambda^nz)| \\
&=& \max_{|w|\leq |\lambda|^n } | h^{(nd)}(w)| \\
&\geq & \max_{|w|\leq |\lambda|^{n/4}} |h^{nd)}(w)|\\
&\geq& |\lambda|^{-dn^2/4}  \max_{|w|\leq |\lambda|^{n/4}}|h(w)|\\
&=& |\lambda|^{-dn^2/4}  p_n(h).
\end{eqnarray*}
 
By setting $C_n=c^n |\lambda|^{\frac{dn(n-1)}{2}}| \lambda|^{-dn^2/4} $ and using the above inequality in equation (\ref{nuevotres}), for $h\in N_{m_n}$ we obtain  
$$
p_0(L^nh)\geq c^n |\lambda|^{\frac{dn(n-1)}{2}}|\lambda|^{-n^2/4}  p_n(h) =C_n p_n(h)
$$
with $C_n\to \infty$.  Thus, by Theorem \ref{not}, $L$ has no  hypercyclic subspaces.
 \end{proof}

\section{Supercyclic extended eigenoperators of $D$}
\label{seccioncuatro}

Let $L$ be an operator on $H(\C)$ satisfying $DL=\lambda LD$ with $\lambda\neq 1$. In this section we  study when $L$ is  supercyclic. 
In this direction, it was proved in  \cite[Theorem 3.6]{bernalbonillacalderon} that no multiple of $C_{\lambda,b}$ is supercyclic for $\lambda\neq 1$. The study of hypercyclicity of the  extended $\lambda$-eigenoperators of $D$ defined by: $T_{\lambda,b}f= f'(\lambda z+b)$ were  studied in \cite{aron,ferandezhallack,pilar}. However, the supercyclicty of such operators, as far we know, has not been studied. 
Our second main result solves this question and provides a complete answer to Question 2 in the introduction. Surprisingly enough, the characterization depends of the value of $\phi$ at the origin.

Let us denote by $HC$ (and $SC$) the subset of all hypercyclic operators (supercyclic operators) defined on $\HC$.
\medskip

{\bf Main result 2. }{\it Assume that $L=R_\lambda \phi(D)$ is an extended $\lambda$-eigenoperator of $D$. The following conditions are equivalent:
\begin{enumerate}
    \item $L\in SC\setminus HC$.
    \item $|\lambda|<1$, $\phi(0)=0$.
\end{enumerate}
}

We point out  that if $L=R_\lambda \phi(D)$ with $\lambda\neq 1$ and $\phi$ has no zeros, then (see Proposition 2.3 \cite{jfa}) $L$ is a multiple of $C_{\lambda,b}$ for some $b\in \mathbb{C}$, hence by Bernal-Bonilla-Calderón \cite{bernalbonillacalderon}, $L$ is not supercyclic. According to the results in \cite{jfa}, the only extended $\lambda$-eigenoperator which are not hypercyclic and could be supercyclic are the extended $\lambda$-eigenoperators associated to $|\lambda|<1$. Thus, we can concentrate on the cases with $|\lambda|<1$ and $\phi^{-1}\{0\}\neq \emptyset$. 

The proof of the second main theorem splits into two cases. First we analyze the case $\varphi(0)=0$ and $|\lambda|<1$, which includes the  Aron-Markose operators. And then we study the case $|\lambda|<1$ and $\phi(0)\neq 0$, in which surprisingly the extended $\lambda$-eigenoperators of $D$ are not supercyclic. 

\begin{theorem}
\label{super1}
If $L=R_\lambda \phi(D)$ with $|\lambda|< 1$ and $\phi(0)=0$, then $L$ is supercyclic.
\end{theorem}
\begin{proof}
It is enough to find a sequence of real numbers $(\lambda_n)$ such that the sequence of operators $(\lambda_nL^n)$ satisfies the conditions of Theorem \ref{HPforseq}.

We write $\phi(z)= z^m\psi(z)$ with $\psi(0)\neq 0$, and set $A_\lambda=R_\lambda \psi(D)$, so that  $L=A_\lambda D^m$. Since $\psi(0)\neq 0$, the polynomials of degree less or equal than $n$ form an invariant subspace for $A_\lambda$. Moreover, the operator $A_\lambda$ has a triangular matrix representation with diagonal entries $\psi(0)\lambda^k$, $k=0,1,\cdots, n$; hence the eigenvalues of the matrix are simple. Let $p_k$ ($0\leq k\leq n$) be the polynomial of degree $k$ 
such that $A_\lambda p_k=\psi(0)\lambda^k p_k$ for $k\geq 0$.
Since $DR_\lambda=\lambda R_\lambda D$ we get:
\begin{eqnarray*}
L^k &=&R_\lambda D^m \psi(D)\cdots R_\lambda D^m \psi(D)  \\
&=&
\lambda^{m} \lambda^{2m}\cdots \lambda^{(k-1)m} R_\lambda \psi(D)\cdots R_\lambda \psi(D) D^{km}  \\
&=& \lambda^{m} \lambda^{2m}\cdots \lambda^{(k-1)m} A_\lambda^kD^{km}.
\end{eqnarray*}

Set $\lambda_k=(\lambda^m\cdots \lambda^{m(k-1)})^{-1}$ and let us denote $T_k=\lambda_kL^k$. 
Since $L^k p(z)=0$ for $km> \textrm{deg}(p)$, the sequence $(\lambda_kL^k)$ converges pointwise to zero on the set of polynomials $X_0=\textrm{linearspan}\{p_k(z)\,\,:\,k\geq 0\}$, which is dense in $\HC$.
Let $V$ be the complex Volterra operator defined by
$$
Vf(z)=\int_0^z f(\xi)d\xi,\quad (z\in \mathbb{C}).
$$

Since $LV^mp_k=A_\lambda p_k=\psi(0)\lambda^k p_k$, we can define
$$
S_kp_n =\frac{V^{mk}p_n}{(\psi(0)\lambda^n)^k },
$$
and we extend $S_k$ to $X_0$ by linearity. It is easy to check that $\lambda_kT_kS_k= \textrm{Id}_{X_0}$.

Finally, for $n_0$ fixed, $S_k p_{n_0}(z) \to 0$ uniformly on compact subsets; hence $S_k$ converges pointwise to zero on $X_0$. Thus, by Theorem \ref{HPforseq}, there exists $f\in \HC$ such that $\{\lambda_kL^kf\}_{k\geq 0}$ is dense in $\HC$, hence $L$ is supercyclic.
\end{proof}

The proof of the next result completes the characterization of supercyclicity for extended $\lambda$-eigenoperators of $D$.
We highlight here the difficulty of proving the non-supercyclicity of an operator.
\begin{theorem}
\label{super2}
Assume that $|\lambda|<1$ and $\phi(0)\neq 0$ then the operator $L=R_\lambda \phi(D)$ is not supercyclic.
\end{theorem}
\begin{proof}
Assume that $f\in \HC$ is supercyclic for $L$.
Thus, given
$g(z)=e^z$, there exists an increasing subsequence of natural numbers $\{n_k\}$ and a sequence of scalars $\{\lambda_{n_k}\}\subset \mathbb{C}\setminus \{0\}$ such that
$$
\lambda_{n_k} (L^{n_k}f)(z)\to e^z \quad \textrm{as}\quad n_k\to \infty
$$
uniformly on compact subsets of $\mathbb{C}$. Moreover, since $D^me^z=e^z$, for each $m\in \mathbb{N}$
$$
\lambda_{n_k} D^m (L^{n_k}f)(z)\to e^z \quad \textrm{as}\quad n_k\to \infty
$$
uniformly on compact subsets of $\mathbb{C}$. In particular,
\begin{equation}
\label{limite1}
\lambda_{n_k} (D^m L^{n_k}f)(z_0)\to e^{z_0} \quad \textrm{and}\quad 
\lambda_{n_k}  (L^{n_k}f)(0)\to 1 \quad \textrm{as}\quad n_k\to \infty.
\end{equation}

From (\ref{limite1}) we get that 
\begin{equation}
\label{cociente}
    \frac{(D^m L^{n_k}f)(z_0)}{(L^{n_k}f)(0)}\to e^{z_0}\quad \textrm{as}\quad  \textrm{as}\quad n_k\to\infty.
\end{equation}
This unusual idea of applying $D^m$ to the previous sequence is crucial and central in the proof. It really simplifies everything.
We will show that we can choose $m\in \mathbb{N}$ such that
\begin{equation}
\label{obje}
    \frac{(D^m L^{k}f)(z_0)}{(L^{k}f)(0)}\rightarrow 0\quad k\to\infty,
\end{equation}
which contradicts (\ref{cociente}). 

To show (\ref{obje}) we will use two different integral representations of the operator $L$.  It was shown in \cite[Theorem 3.3]{jfa} that there exists a Borel measure $\mu$ with compact support in $\mathbb{C}$ such that
$$
(L^kf)(z)=\int\cdots\int f\left(\lambda^kz+\lambda^{k-1}w_1+\cdots+w_k\right)d\mu(w_k)\cdots d\mu(w_1).
$$
Now, since the support of $\mu$ is contained in some disc $D(0,R)$ and 
$$
|\lambda^k z+\lambda^{k-1}w_1+ \cdots+w_k|\leq M(|z|)= |\lambda|^k|z|+ \frac{1-|\lambda|^k}{1-|\lambda|}R
$$
for $|z|<r$, we get that each element in the argument of $f$ lies in the disk $D(0,M(r))$. Therefore, for $f\in \HC$, if $|z|\leq r$ then 
$$
|L^kf(z)|\leq \sup_{|z'|=M(r)} |f(z')| \|\mu\|^k,
$$
where $\|\mu\|$ denotes the total variation of the measure $\mu$.
On the other hand, since $DL=\lambda LD$, for $|z|\leq r$ we get 
\begin{equation}
    \label{superiorr}
    |(D^mL^kf)(z)|=|\lambda|^{mk} |(L^k)D^mf(z)|\leq |\lambda|^{mk} \|\mu\|^k \max_{|z|=M(r)} |D^mf(z)|.
\end{equation}

Next we obtain lower estimates of $|(L^{k}f)(0)|$. Indeed, let us observe that if $\omega=\lambda^{-1}$ then
\begin{eqnarray}
\label{iterada}
L^k&=&R_\lambda \phi(D)R_\lambda \phi(D)\cdots R_\lambda\phi(D)\\
&=& \phi(\omega D)\phi(\omega^2D)\cdots\phi(\omega^k D) R_\lambda^k.
\end{eqnarray}

We consider the sequence of entire functions $\Phi_k(z)=\phi(\omega z)\cdots \phi(\omega^kz)$ with series expansions  
\begin{equation}
    \label{serie}
    \Phi_k(z)=\sum_{m=0}^\infty \frac{a_m^{(k)}}{m!} z^m,
\end{equation}
where $a_m^{(k)}= \Phi^{(m)}(0)$ can be obtained by Leibnitz's rule. Therefore 
\begin{eqnarray*}
   \left| \Phi_k^{(m)}(0)\right| &=& 
\left|[\phi(\omega z),\ldots, \phi(\omega^k z)]^{(m)}(0) \right|\\
&=& \left|\sum_{h_1+\ldots+h_k=m} \binom{m}{h_1,\ldots,h_k} \prod_{t=1}^{k} (\phi(\omega^{t} z))^{(h_t)}(0)\right|\\
&\leq & \sum_{h_1+\ldots+h_k=m}
\binom{m}{h_1,\ldots,h_k} \prod_{t=1}^{k} |\omega^{th_t}(\phi^{(h_t)}(0))|\\
&\leq &{\rm C}^m (|\omega|+\ldots+|\omega|^k)^{m},
\end{eqnarray*}
with $C=\sup_{t\in\mathbb{N}\cup\{0\}}|\phi^{(t)}(0)|$. Thus, for each  $f\in \HC$  
$$
L^kf(z)= \sum_{m=0}^\infty \frac{a_m^{(k)}}{{m!}} D^m R_{\lambda}^kf(z) =  \sum_{m=0}^\infty \frac{a_m^{(k)}}{{m!}} \lambda^{km} f^{(m)}(\lambda^k z).
$$

Using the estimates of $a_m^{(k)}$ previously obtained we get:
$$
|L^kf(0)|= \sum_{m=0}^\infty  \frac{|a_m^{(k)}|}{{m!}}| \lambda^{km}| |f^{(m)}(0)|
\leq  \sum_{m=0}^{\infty}\frac{K^mC^m}{m!} |f^{(m)}(0)|,
$$
with $K$ an upper bound of the bounded sequence $(|\lambda|^{k} [|\omega|+\cdots+|\omega|^{k}])_k$.

Thus, if $f=\sum_{m=0}^\infty \frac{f^{(m)}(0)}{m!} z^m$ is a supercyclic vector for $L$, we decompose $f(z)=(P_{m_0}f)(z)+(T_{m_0}f)(z)$, here $(P_{m_0}f)(z)$ is the Taylor polynomial of degree $m_0$ and $(T_{m_0}f)(z)$ is the tail of the series.
By the above estimates we known that there exist $m_0$ which depends only on $f$ such that
\begin{equation}
\label{restor}
|L^k(T_{m_0}f)(0)|\leq 1/4
\end{equation}
uniformly on $k$.

On the other hand,  we consider $L$ acting on the invariant subspace $\{1,z,\cdots,z^{m_0}\}$. On such invariant subspace $L$ is triangular with diagonal entries $\phi(0)\lambda^k$, $k=0,1,\cdots,m_0$. Since supercyclicity is invariant by multiplication by nonzero scalars, we can suppose without loss that $\phi(0)=1$.
Since the eigenvalues are different, $\{1,p_1(z),\cdots, p_{m_0}(z)\}$ is a basis of the space $\mathcal{P}_{m_0}[z] $ of polynomials of degree less or equal than $m_0$.
Therefore, there exist scalars $c_0,c_1,\cdots,c_{m_0}$ such that $(P_{m_0}f)(z)=c_0 1+c_1 p_1(z)+\cdots+c_{m_0} p_{m_0}(z)$.  Moreover, if $L$ were supercyclic, then there would be a dense subset of supercyclic vectors. Thus, 
we can suppose without loss that $c_0\neq 0$. Moreover since a non-zero multiple of a supercyclic vector is also supercyclic, we can suppose that $c_0=1$.
Therefore
$$
L^k(P_{m_0}f)(z)=1+c_1\lambda^k p_1(z)+\cdots+ c_{m_0}\lambda^{m_0k} p_{m_0}(z),
$$
which implies that there exist $k_0$ such that fo any $k\geq k_0$,
$$|c_1\lambda^k p_1(0)+\cdots+ c_{m_0}\lambda^{m_0k} p_{m_0}(0)|<1/4.$$ 
That is:
\begin{equation}
\label{cabezar}
|L^{k}P_{m_0}f(0)|\geq 3/4
\end{equation}
for all $k\geq k_0$. Thus, by selecting $m$ such that 
$|\lambda|^m<\|\mu\|$ and using the estimates (\ref{superiorr}), (\ref{restor}) and (\ref{cabezar}), we get:

\begin{eqnarray*}
\left|     \frac{(D^m L^{k}f)(z_0)}{(L^{k}f)(0)} \right| &=&
\left|   \frac{(D^m L^{k}f)(z_0)}{| L^k P_{m_0}f (0)+L^kT_{m_0}f(0)| }\right| \\
&\leq & 
\frac{|\lambda|^{mk} \|\mu\|^k \max_{|z|=M(r)} |D^mf(z)|}{3/4-1/4}
\rightarrow 0
\end{eqnarray*}
as $k\to \infty$, which yields the desired result.
\end{proof}

\section{Supercyclic subspaces}
\label{seccioncinco}

Let us donote by $HC_\infty$ (respectively $SC_\infty$) the subset of operators having a closed infinite dimensional subspace whose non-zero elements are hypercyclic (respec. supercyclic).
In this section we characterize the extended $\lambda$-eigenoperators of $D$ which belong to the subset $SC_\infty\setminus HC_\infty$. 

If $|\lambda|=1$ we know that $L=R_\lambda\phi(D)$ has a hypercyclic subspace. Also, if $0<|\lambda|<1$ and $\phi(0)\neq 0$ then $L=R_\lambda\phi(D)$ is not supercyclic.  Finally, if $|\lambda|>1$ and $\phi$ has infinite zeros then $L=R_\lambda\phi(D)$ has a hypercyclic subspace. Thus, the results in Sections \ref{secciontres} and \ref{seccioncuatro} allow us to focus the study  in the following cases: 1) $0<|\lambda|<1$ and $\phi(0)=0$; and  2) $|\lambda|>1$ and $\phi$ has a finite number of zeros. We fully cover both cases obtaining the following characterization:
\medskip

{\bf Main result 3.} {\it Let $L$ be an extended $\lambda$-eigenoperator of $D$, $\lambda\neq 1$. The following conditions are equivalent:
\begin{enumerate}
\item $L\in SC_\infty\setminus HC_\infty$.
\item $0<|\lambda|<1$ and $\phi(0)=0$.
\end{enumerate}
}

The proof of the next result uses some of the ideas applied in Proposition \ref{infinitosceros} and Theorem \ref{modulo1}.

\begin{theorem}
\label{modulomenor1}
Assume that $L=R_\lambda \phi(D)$.
If $0<|\lambda|< 1$ and $\phi(0)=0$, then $L\in SC_\infty\setminus HC_\infty$.
\end{theorem}
\begin{proof} 
 If $\phi(0)=0$ then $\phi(z)=z^m\psi(z)$ with $\psi(0)\neq 0$. By taking $\lambda_k=(\lambda^m \lambda^{2m}\cdots \lambda^{(k-1)m})^{-1}$ as in the proof of Theorem \ref{super1}, we get that the sequence  $(\lambda_kL^k)$ satisfies the Hypercyclicity Criterion (Theorem \ref{HPforseq}).  
 
If $\phi$ has infinitely many zeros, then  $M_0=\textrm{ker}(L)$ is an infinite dimensional closed subspace and, for  $f\in M_0$,  $\lambda_{n}L^{n}f\to 0$ uniformly on compact subsets.
Then by Theorem \ref{yes} we obtain that the sequence $\lambda_nL^n$ has a hypercyclic subspace. Hence $L$ has a supercyclic subspace.

Now, let us suppose that $\phi$ has a finite number of zeros and $0<|\lambda|<1$. By Proposition \ref{semejanza} we can suppose without loss that  $L=R_\lambda P(D)$ with $P$ a non-constant polynomial.  

Now, we will rescue some ideas of Theorem \ref{modulo1}. 
We set $\omega=\lambda^{-1}$, and we denote $P(z)=\sum_{k=0}^d b_kz^k$.  For each $n\in\N$ we write  
$$
P(\omega z)P(\omega^2 z)\cdots P(\omega^{n}z)=\sum_{k=0}^{nd}b_k^{(n)} z^k.
$$
Let us denote $\widetilde{C_n}=\max\{|b_k^{(r)}| \,:\, 1\leq r\leq n\,,\, 1\leq k\leq r\}$. 
Then
$$
\lambda_n(R_\lambda P(D))^n (z^s)= \sum_{k=0}^{nd} b_k^{(n)} [s(s-1)\cdots (s-k+1)] \frac{\lambda^{ns}}{\lambda^{m}\cdots \lambda^{(n-1)m}} z^{s-k}.
$$
Thus, for $|z|=M$, we get  
$$
|\lambda_n (R_\lambda P(D))^n (z^s)|\leq  \widetilde{C_n} nd  s^{nd} \left| \frac{\lambda^{ns}}{\lambda^{m}\cdots \lambda^{(n-1)m}}\right|M^s.
$$


For each $n$ we can find $M_n>0$ such that, for $x\geq M_n$,  
$$
 nd x^{dn}\leq 2^x \quad \textrm{ and }\quad 
\widetilde{C_n}\left|\frac{\lambda^{nx}}{\lambda^{m}
\cdots \lambda^{(n-1)m}}\right|\leq 1.
$$
Thus, we can construct inductively a strictly increasing sequence $(n_j)$ in $\N$ such that $n_{j+1}\geq M_{n_j}$ and, for $s> j$, 

\begin{equation}
    \label{uno-dos}
    n_j d n_s^{n_jd} \leq 2^{n_s}\quad 
    \textrm{and}\quad 
    \widetilde{C_{n_j}}  \left|\frac{\lambda^{n_jn_s}}{\lambda^{m}
\cdots \lambda^{(n_j-1)m}}\right|\leq 1.
\end{equation}

Hence, using the  estimates (\ref{uno-dos}),  if $f\in N_j= \overline{\textrm{linearspan}} \{z^{n_s}\,: \,s> j\}$ we have $f(z)=\sum_{s=j+1}a_s z^{n_s}$ and 
\begin{eqnarray*}
\rho_M(\lambda_{n_j}L^{n_j} f)&=& \rho_M\left(L^{n_j}\left(\sum_{s=j+1}^{\infty} a_s z^{n_s}\right) \right)\\
&\leq &  \sum_{s=j+1} |a_s| (2M)^{n_s}=  \rho_{2M}(f).
\end{eqnarray*}

Thus, the conditions of Theorem \ref{yes2} are fulfilled if we consider the sequence of seminorms $\rho_n(f)$, and the sequences  $m(n)=2n$ and $C_n=k(n)=1$. Since the sequence of operators $(\lambda_nL^n)$ satisfies the Hypercyclicity Criterion (see Theorem \ref{super1}), $(\lambda_nL^n)$ has a hypercyclic subspace. Hence $L$ has a supercyclic subspace as we desired to prove.
\end{proof}

Now, we turn our attention to the case $L=R_\lambda \phi(D)$ with $|\lambda|>1$ and $\phi$ has a finite (non empty) number of zeros. We proved in Theorem \ref{modulomayoruno} that in such a case all closed subspaces of hypercyclic vectors for $L$ has finite dimension. And the question now is the following: by relaxing hypercyclicity by supercyclicity, could we obtain an infinite dimensional closed subspace $M_1$ such that $x\in M_1\setminus\{0\}$ is supercyclic for $L$?

We discard the case in which $\phi(z)$ has no zeros, because in such a case  $L=R_\lambda \phi(D)$ is a multiple of the composition operator $C_{\lambda,b}f(z)=f(\lambda z+b)$, which is not supercyclic.
Thus we focus our attention in the case that $\phi$ has a finite (nonempty) number of zeros, and by Proposition \ref{semejanza} again we can suppose that $L=R_\lambda P(D)$ with $P$ a non constant polynomial.

Let us point out that proving that an operator is not supercyclic is a more complicated task than proving that it is, as we can see in the work of A. Montes and H.N. Salas (\cite{montessalas}), the proofs of the non-existence of supercyclic subspaces are even more sophisticated. In our case, the refinements are enhanced by Fréchet spaces context and the structure of the involved operators.

We will use the following lemma, whose proof can be founded in \cite[Lemma 10.39]{erdmannperis}, and we refer to \cite{crelle} for a proof in the Banach space setting.

\begin{lemma}
\label{mullertric}
Let $X$ be a Fréchet space, $F$ a finite-dimensional subspace of $X$, $\rho$ a continuous seminorm on $X$ and $\varepsilon>0$. Then there exists a closed subspace $H$ of finite codimension such that for any $x\in F$ and $y \in H$ 
$$
\rho(x+y)\geq \max\left\{ \frac{\rho(x)}{1+\varepsilon},\frac{\rho(y)}{2+\varepsilon}\right\}.
$$
\end{lemma}


Our argument hinges on the following computations.

\begin{lemma}
\label{infinf}
Assume that $m>nd$. If $N_n=\{f\in\HC\,:\,f^{k)}(0)=0, k\leq m-1\}$, then for any $f\in N_n$: 
\begin{equation}
    \label{ecinf}
    \rho_1(f^{(nd)}(\lambda^nz))\geq \frac{m!}{(m-nd)!} |\lambda|^{m\frac{n}{2}-n^2d} \rho_{|\lambda|^{n/2}}(f)
\end{equation}
\begin{proof}
If $f\in N_m$, $f(z)= \sum_{p=m}^\infty a_pz^p$, then
\begin{eqnarray*}
\rho_1(f^{(nd)}(\lambda^nz)&=&\sum_{p=m}^\infty  |a_p|p(p-1)\cdots (p-nd+1) |\lambda|^{n(p-nd)}\\
&\geq & \frac{1}{|\lambda|^{n^2d}} \sum_{p=m}^\infty |a_p| \frac{p!}{(p-nd)!} (|\lambda|^{n/2}|\lambda|^{n/2})^{p} \\
&\geq & \frac{|\lambda|^{m n/2}}{|\lambda|^{n^2d}}\frac{m!}{(m-nd)!} \rho_{|\lambda|^{n/2}}(f)
\end{eqnarray*}
which proves inequality (\ref{ecinf}) as we desired.
\end{proof}
\end{lemma}

\begin{lemma}
\label{supsup}
Assume that $L=R_\lambda P(D)$, with $d=\textrm{grad}(P)$.
For each $f\in \HC$,
there is a constant $B>0$ such that
\begin{equation}
    \label{ecsup}
    |L^nf(0)|\leq B^n \frac{((d+1)n-1)!}{(n-1)!} |\lambda|^d|\lambda|^{2d}\cdots |\lambda|^{(n-1)d}\rho_1(f).
\end{equation}
\end{lemma}
\begin{proof}
Observe that
$$
L^n=R_\lambda P(D)\cdots R_\lambda P(D)=\widehat{P_n(D)}R_\lambda^n,
$$
and $\widehat{P_n(z)}=P(\omega z)\cdots P(\omega^nz)$. If we write  $P(z)=\sum_{k=0}^dp_kz^k$ and denote  $B=\max_{k}|p_k|$, then
$$
\widehat{P_n}(z)=\sum_{m=0}^{nd}\frac{a_m^{(n)}}{m!}z^m,
$$
and using Leibnitz's formula we obtain 
$$
a_m^{(n)}=\widehat{P_n}^{m)}(0)=\sum_{h_1+\cdots +h_n=m}\binom{m}{h_1\cdots h_n}\prod_{t=1}^n(\omega^{th_t}p_{h_t}h_t!),
$$
from which we get 
\begin{eqnarray*}
|L^nf(0)|&\leq &\sum_{m=0}^{nd} \frac{1}{m!} \sum_{h_1+\cdots +h_n=m}\binom{m}{h_1\cdots h_n}\prod_{t=1}^n(|\omega|^{th_t}|p_{h_t}|h_t!)|D^mR_{\lambda}^nf(0)|     \\
&\leq & \sum_{m=0}^{nd} \frac{1}{m!} \sum_{h_1+\cdots +h_n=m} m! |\lambda|^{nm-h_1-\cdots-nh_n} B^n |f^{(m)}(0)|\\
&\leq & B^n |\lambda|^d\cdots |\lambda|^{(n-1)d}\sum_{m=0}^{nd} \binom{m+n-1}{m}|f^{(m)}(0)|\\
&\leq & B^n |\lambda|^d\cdots |\lambda|^{(n-1)d} \frac{(nd+n-1)!}{(n-1)!}\sum_{m=0}^{nd} \frac{|f^{(m)}(0)|}{m!}\\
&\leq &B^n |\lambda|^d\cdots |\lambda|^{(n-1)d} \frac{(nd+n-1)!}{(n-1)!} \rho_1(f).
\end{eqnarray*}
as promised.
\end{proof}

Now let us prove the main result that complete the Table \ref{tabla3}. 

\begin{theorem}
\label{general}
Assume that $L=R_\lambda p(D) $ and $|\lambda|>1$. Then all closed subspaces of supercyclic vectors for $L$ have finite dimension.
\end{theorem}
\begin{proof}
We want to show that in any infinite dimensional closed subspace $M\subset \HC$  there exists $f\in M\setminus\{0\}$ such that 
\begin{equation}
\label{objetivofinal}
\frac{|L^nf(0)|}{\rho_{1}(L^nf)} \to 0\quad \textrm{as}\quad n\to \infty.
\end{equation}
Thus $f$ is not supercyclic for $L$, hence $L$ has no supercyclic subspace.

The sequence of seminorms   $p_n(\cdot)=\rho_{|\lambda|^; {n/2}} (\cdot)$ induces the topology of $\HC$. 
Let  $N_n=\{f\in\HC\,:\,f^{k}(0)=0; \, k=1,\cdots, n-1\}$, and let $(m_n)_n$ be the increasing sequence obtained in Lemma \ref{iteracionpolinomio}.Thus there is a constant $c>0$ such that for each $h\in N_{m_n}$:
\begin{equation}
    \label{siguiente}
\rho_1(L^nh)\geq c^n 1\cdot |\lambda|^d\cdot |\lambda|^{2d}\cdots |\lambda|^{(n-1)d} \rho_1(h^{nd)}(\lambda^n z)).
\end{equation}

By Remark \ref{observacionpolinomio}, we can assume without loss that $m_n>dn$, otherwise we select a convenient subsequence of $(m_n)$. The idea of the proof is to construct a sequence $f_n\in  M\cap N_{m_n}$ satisfying 
$p_n(f_{n})=\frac{1}{n^2}$ for $n\geq 0$ and 
$L^nf_n\in H_{n-1}$ for $j=1,...,n$. 

Here $H_{n-1}$ is the finite codimensional subspace guaranteed by Lemma \ref{mullertric} associated to $F_{n-1}=\textrm{span} \{T^jf_i\,:i\leq n-1,j\leq n\}$.

Indeed, we take $f_1\in M\setminus \{0\}$. Then  $f_1$ is supercyclic for $L$, and since $p_1(f_1)\neq 0$, we can suppose that $p_1(f_1)=1$.
Let us consider  the subspace $F_1=\textrm{span}\{T^{j}f_1\,:\, j=1,2\}$ of finite dimension and the corresponging subspace $H_1$  of finite codimension guaranteed by Lemma \ref{mullertric}. 

Assume that $f_1,\cdots,f_{n-1}$ have been already been constructed. Let us consider
$F_{n-1}=\textrm{span}\{T^jf_i\,:i\leq n-1,j\leq n\}$ and by $H_{n-1}$ the corresponding finite codimension subspace guaranteed by Lemma \ref{mullertric}.

Since $N_{m_n}$ is of finite codimension, there exists $f_n\in M\setminus\{0\}\cap N_{m_n}$, such that $T^nf_{n}\in H_{n-1}$. Again $f_n$ is supercyclic and $p_n(f_n)\neq 0$, so we can suppose that $p_n(f_n)=\frac{1}{n^2}$.

We claim that $f=\sum_{k=1}^\infty f_k$ is the function in $M\setminus\{0\}$ that we are looking for. Indeed, since $f_n\in M\setminus\{0\}$, $p_n(f_n)=1/n^2$ and $(p_n)$ is an increasing sequence of seminorms defining the topology of $\HC$, the series
$\sum_{k=1}^\infty f_k$ defines an entire  function  $f\in M\setminus\{0\}$.

Now, with this construction in hand, we estimate the denominator  of (\ref{objetivofinal}).
First we show that  $\rho_1(L^nf)\geq \frac{1}{(\varepsilon+1) (\varepsilon+2)} \rho_1(L^nf_n)$.  Indeed

\begin{eqnarray}
\notag
\rho_1(L^nf)&=& \rho_1\left(L^n\left(\sum_{j=1}^{n} f_j\right)+L^n\left(\sum_{j=n+1}^\infty f_j\right)\right) \\
&\geq &\frac{1}{1+\varepsilon}\rho_1\left(L^n\left(\sum_{j=1}^{n-1} f_j\right)+L^nf_n\right) \label{cola}\\
&\geq &\frac{1}{1+\varepsilon} \frac{1}{2+\varepsilon}
\rho_1\left(L^nf_n\right) \label{cabeza}
\end{eqnarray}

For inequality (\ref{cola}) we have applied Lemma \ref{mullertric}: $L^n(\sum_{j=1}^nf_j)\in F_n$ and $L^n(\sum_{j=k+1}^\infty f_j)\in H_{n}$. And (\ref{cabeza}) follows the same way:
$L^n(\sum_{j=1}^{n-1}f_j)\in F_{n-1} $ and $L^n f_n\in H_{n-1}$. 

Next, since $f_n\in M_{m_n}$,  applying (\ref{siguiente}) to  (\ref{cola}) and (\ref{cabeza}) we get  
\begin{eqnarray}
\notag
\rho_1(L^nf)&\geq &\frac{1}{1+\varepsilon} \frac{1}{2+\varepsilon}
\rho_1\left(L^nf_n\right)\\
&\geq & \frac{c^n 1\cdot |\lambda|^d\cdot |\lambda|^{2d}\cdots |\lambda|^{(n-1)d}}{(1+\varepsilon)(2+\varepsilon)} \rho_1(f_n^{nd)}(\lambda^n z)).\label{penultimo}
\end{eqnarray}

According to Proposition \ref{infinf}, since $f_n\in M_{m_n}$,
we incorporate inequality (\ref{ecinf}) into inequality (\ref{penultimo}) and we get 
\begin{eqnarray}
\notag
\rho_1(L^nf)&\geq & \frac{c^n 1\cdot |\lambda|^d\cdot |\lambda|^{2d}\cdots |\lambda|^{(n-1)d}}{(1+\varepsilon)(2+\varepsilon)} \frac{|\lambda|^{m_n n/2}}{|\lambda|^{n^2d}}\frac{m_n!}{(m_n-nd)!} \rho_{|\lambda|^{n/2}}(f_n) \\
&=& \frac{c^n 1\cdot |\lambda|^d\cdot |\lambda|^{2d}\cdots |\lambda|^{(n-1)d}}{(1+\varepsilon)(2+\varepsilon)} \frac{|\lambda|^{m_n n/2}}{|\lambda|^{n^2d}}\frac{m_n!}{(m_n-nd)!}\frac{1}{n^2}. \label{final1}
\end{eqnarray}

Now we estimate the numerator of $(\ref{objetivofinal})$.
Since $m_n>dn$, we deduce that $L^nf_j(0)=0$ for all $j\geq n$. Therefore, according to Lemma \ref{supsup}, we get 
\begin{eqnarray}
\notag
    |L^nf(0)|&=&\left|L^n\left(\sum_{j=1}^{n-1}f_j\right)(0)\right|\\\notag
    &\leq & B^n \frac{((d+1)n-1)!}{(n-1)!} |\lambda|^d|\lambda|^{2d}\cdots |\lambda|^{(n-1)d}\rho_1\left(\sum_{j=1}^{n-1}f_j\right)\\ \notag
    &\leq & B^n \frac{((d+1)n-1)!}{(n-1)!} |\lambda|^d|\lambda|^{2d}\cdots |\lambda|^{(n-1)d}\sum_{j=1}^{n-1}p_j(f_j) \\ 
    &\leq &\frac{\pi^2}{6} B^n \frac{((d+1)n-1)!}{(n-1)!} |\lambda|^d|\lambda|^{2d}\cdots |\lambda|^{(n-1)d}.\label{final2}
\end{eqnarray}

Therefore, if we compute (\ref{objetivofinal}) using the inequalities (\ref{final1}) and (\ref{final2}) we obtain:
\begin{eqnarray}
\label{FIN}
\frac{|L^nf(0)|}{\rho_{1}(L^nf)}  &\leq & \frac{\frac{\pi^2}{6} B^n(1+\varepsilon)(2+\varepsilon)n^2}{c^n |\lambda|^{\frac{nm_n}{2}-dn^2}},
\end{eqnarray}

Choosing a subsequence of $(m_n)$ if necessary, we get that the sequence (\ref{FIN}) converges to zero, as promised.
\end{proof}

\end{document}